\documentclass[11pt]{amsart}   
\usepackage[greek,english]{babel}  
\usepackage{graphicx}
\usepackage{amsmath}
\usepackage{amssymb}
\usepackage{amsmath,amsthm}
\usepackage{verbatim}
\usepackage{mathtools}
\usepackage{tikz}
\usepackage{soul}
\newtheorem{theorem}{Theorem}[section]
\newtheorem{definition}[theorem]{Definition}

\newtheorem{lemma}[theorem]{Lemma}

\newtheorem{proposition}[theorem]{Proposition}

\theoremstyle{definition}

\newcommand{\RR}{\mathbb{R}}

\titlepage 
\title[Periodic solutions to a forced Kepler problem in the plane]
{Periodic solutions to a forced Kepler problem in the plane}
\author{Alberto Boscaggin, Walter Dambrosio and Duccio Papini}
\address{Alberto Boscaggin and Walter Dambrosio\newline \indent
 Dipartimento di Matematica ``Giuseppe Peano'', \newline \indent
Universit\`a di Torino, \newline \indent
Via Carlo Alberto, 10,
10123 Torino, Italy \newline \newline \indent
Duccio Papini \newline \indent
 Dipartimento di Scienze Matematiche, Informatiche e Fisiche, \newline \indent
Universit\`a di Udine, \newline \indent
Via delle Scienze, 206,
33100 Udine, Italy
}

\email{alberto.boscaggin@unito.it}
\email{walter.dambrosio@unito.it}
\email{duccio.papini@uniud.it}
\date{}

\begin{document}

\begin{abstract}
Given a smooth function $U(t,x)$, $T$-periodic in the first variable and satisfying $U(t,x) = \mathcal{O}(\vert x \vert^{\alpha})$ for some $\alpha \in (0,2)$ as $\vert x \vert \to \infty$, we prove that the forced Kepler problem
$$
\ddot x = - \dfrac{x}{|x|^3} + \nabla_x U(t,x),\qquad x\in \RR^2,
$$
has a generalized $T$-periodic solution, according to the definition given in the paper [Boscaggin, Ortega, Zhao, \emph{Periodic solutions and regularization of a Kepler problem with time-dependent perturbation}, Trans. Amer. Math. Soc, 2018]. The proof relies on variational arguments.
\end{abstract}

\date{\today}
\keywords{Kepler problem; periodic solutions; collisions; variational methods.}
\subjclass{37J45, 70B05, 70F16. }

\thanks{{\bf Acknowlegments.} Work partially supported by the 
ERC Advanced Grant 2013 n. 339958
{\it Complex Patterns for Strongly Interacting Dynamical Systems - COMPAT}, by the INDAM-GNAMPA Projects
\textit{Dinamiche complesse per il problema degli $N$-centri} and \textit{Propriet\`a qualitative di alcuni problemi ai limiti}
and by the project PRID \textit{SiDiA – Sistemi Dinamici e Applicazioni} of the DMIF - Universit\`a di Udine}

\maketitle
\medbreak

\section{Introduction and statement of the main result}
\setcounter{page}{1}

In this paper we investigate the existence of $T$-periodic solutions to the equation
\begin{equation} \label{eq-kepleromain}
\ddot x = - \dfrac{x}{|x|^3} + \nabla_x U(t,x),\qquad x\in \RR^2,
\end{equation}
where $U: \mathbb{R} \times \mathbb{R}^2 \to \mathbb{R}$ is a (smooth) function $T$-periodic in its first variable (for some $T > 0$).
A special interesting case occurs when $ U(t,x) = \langle p(t), x \rangle $ for some $T$-periodic forcing term
$ p $, which gives rise to the equation
\begin{equation}\label{eq-keplerofor}
\ddot x = - \dfrac{x}{|x|^3} + p(t),\qquad x\in \RR^2.
\end{equation}
As well known, equation \eqref{eq-kepleromain} models the motion of a massless particle $x\in \RR^2$ subject to the action of both the gravitational force and an external force with potential $U(t,x)$; accordingly, it can be meant as a (time-periodically) forced Kepler problem. 
In spite of its simple looking structure, such an equation possesses some peculiar features making it a quite paradigmatic model for the methods of Nonlinear Analysis and Dynamical Systems. In particular, as typical in problems of Celestial Mechanics, 
the possibility for a solution to approach the collision set $\{x = 0\}$ has to be taken into account, leading to substantial difficulties.

To the best of our knowledge, most of the results available up to now have been proved in a perturbative setting, namely, for the equation
\begin{equation} \label{eq-kepleroper}
\ddot x = - \dfrac{x}{|x|^3} + \varepsilon \,\nabla_x U(t,x),\qquad x\in \RR^2,
\end{equation}
with $\varepsilon$ small enough, see \cite{AmbCot89,BosOrt16,CaVi00,FonGal17,FonGal18,FoToTo12} and the references therein. In such a case, classical (i.e., without collisions) $T$-periodic solutions are found, for $\varepsilon$ small enough, near the ones of the unperturbed Kepler problem ($\varepsilon = 0$), via perturbative techniques. Also this situation, however, is far from being trivial, since the peculiar degeneracies of the Kepler problem rule out the possibility of using the standard perturbation theory of completely integrable Hamiltonian systems. As a matter of fact, one is typically led to assume some symmetry conditions on the potential $U(t,x)$, eventually ruling out the simple case of equation \eqref{eq-keplerofor}. 
We also mention the paper \cite{AmHaOrUr11} in which the case $\varepsilon$ large is considered.

As far as equation \eqref{eq-kepleromain} is concerned, some results were given in \cite{SerTer94}.
In that paper, global variational methods are used, requiring the development of delicate action level estimates for solutions approaching the origin. In order for this procedure to work so as to prevent the occurrence of collisions, again some symmetry conditions on the potential are imposed and equation \eqref{eq-keplerofor} is left out from the analysis therein.

Recently, a different point of view has been proposed in \cite{BosOrtZhaPP}, where a suitable definition of \emph{generalized} solution to \eqref{eq-kepleromain} was given.
We recall it below for the reader's convenience.

\begin{definition}\label{def:bs}
A \emph{generalized $T$-periodic solution} to \eqref{eq-kepleromain} is a continuous and $T$-periodic function $x: \mathbb{R} \to \mathbb{R}^2$ satisfying the following conditions:
\begin{itemize}
\item[(i)] the set $E_x := \{t \in \mathbb{R} : x(t) = 0\}$ of collision instants is discrete,
\item[(ii)] for any open interval $I \subset \mathbb{R} \setminus E_x$, the function $x$ is $\mathcal{C}^2(I)$ and satisfies 
\eqref{eq-kepleromain} on $I$, 
\item[(iii)] for any $t_0 \in E_x$, the limits
$$
\lim_{t \to t_0} \frac{x(t)}{\vert x(t) \vert} \qquad \mbox{and} \qquad \lim_{t \to t_0}\left( \frac12 \vert \dot x(t) \vert^2- \frac{1}{\vert x(t) \vert}\right)
$$
exist and are finite.
\end{itemize}
\end{definition}

The possibility of considering solutions attaining the value $x = 0$ was already discussed by various authors (see, for instance, 
\cite{AmbCot93,BahRab89,Tan94}). However, while in these papers a generalized solution is just meant as an $H^1$-function attaining the value $x = 0$ on a zero-measure set (and solving the equation on the complementary set), Definition \ref{def:bs} requires a precise behavior at the collisions instants: that is, both the collision direction $\tfrac{x(t)}{\vert x(t) \vert}$ and the collision energy $\tfrac12 \vert \dot x(t) \vert^2- \tfrac{1}{\vert x(t) \vert}$ are continuous functions. As shown in \cite{BosOrtZhaPP}, this is a very natural definition of solution for equation \eqref{eq-kepleromain}, since it corresponds to the notion of solution provided by the well known Levi-Civita regularization for the planar Kepler problem
(see \cite{Wa2010} for some basic references about the theory of regularization in Celestial Mechanics and \cite{MaOrRe14} for an application of regularization techniques to a Kepler problem with linear
drag). 

Using Levi-Civita regularization together with a delicate bifurcation theory from (fixed-energy) periodic manifolds of autonomous Hamiltonian systems \cite{Wei73}, a universal existence result can be proved for equation \eqref{eq-kepleroper}: precisely, with 
no assumptions (but the smoothness) on the potential $U(t,x)$, a generalized $T$-periodic solution always exists (see \cite[Theorem 3.1]{BosOrtZhaPP} for a more precise statement). 

The aim of this brief paper is to extend such an existence result to a non-perturbative setting. Precisely, we are going to prove the following theorem.

\begin{theorem} \label{teo-main}
Let $U: \mathbb{R} \times \mathbb{R}^2 \to \mathbb{R}$ be a $\mathcal{C}^1$ function, $T$-periodic in the first variable (for some $T > 0$); moreover, suppose that, for some $C > 0$ and $\alpha \in (0,2)$,
\begin{equation}\label{eq-subquadratic}
\vert U(t,x) \vert \leq C (1 + \vert x \vert^{\alpha}), \quad \mbox{ for every } (t,x) \in \mathbb{R} \times \mathbb{R}^2.
\end{equation}
Then, there exists at least one generalized $T$-periodic solution to \eqref{eq-kepleromain}.
\end{theorem}

In particular, a generalized $T$-periodic solution to \eqref{eq-keplerofor} exists, for any $T$-periodic function $p$ of class $\mathcal{C}^1$.
Incidentally, we mention that existence and multiplicity of generalized $T$-periodic and quasi-periodic solutions to the one-dimensional forced Kepler problem
$$
\ddot x = - \dfrac{x}{|x|^3} + p(t),\qquad x\in \RR^+ := [0,+\infty),
$$
was previously investigated in \cite{Ort11,RebSimPP,Zha16}, using the Poincar\'e-Birkhoff fixed point theorem and KAM theory.

The proof of Theorem \ref{teo-main} relies on a variational argument.
This kind of apporach has been used also for other equations in Celestial Mechanics, see e.g. \cite{AmbCot89,AmbCot93,BahRab89,BaFeTe08,BarTerVer14,BosBotDam18,BosDamPap18,BosDamTer17,Cas,FuGrNe,MadVen09,SerTer94,SoaTer12,Tan93b,TeVe,Yu16} and the references therein.
First, in Section \ref{sec2} we minimize the action functional associated with \eqref{eq-kepleromain} on the weak closure of $H^1$-loops with nontrivial winding number around the origin: as well known (see \cite{Go-77}), this topological constraint provides the needed coercivity, so that a minimum exists by the direct method of calculus of variations.
Then, in Section \ref{sec3} we investigate the behavior of the above found minimum near its possible collisions, so as to prove that it corresponds to a generalized $T$-periodic solution according to Definition \ref{def:bs}. The hardest part of this step is to show that the ingoing and outgoing collision directions must coincide (that is, the existence of the first limit in condition (iii)): we prove this via a blow-up analysis, eventually relying on a well-known action level estimate for the direct and indirect Keplerian arc (see Lemma~\ref{prop-archi}).    

\section{Minimizing the action functional}\label{sec2}

In this section we prove the existence of a minimum, in a suitable class of functions, of the action functional associated with \eqref{eq-kepleromain}. 

To this end, for every continuous function $x:[0,T]\to \RR^2\setminus \{0\}$ such that $x(0)=x(T)$, we first denote by $r_x$ the winding number of $x$ around the origin, that is, writing in polar coordinates $x(t) = \rho(t) e^{i\theta(t)}$, with $\rho(t) > 0$, 
$$
r_x = \frac{\theta(T) - \theta(0)}{2\pi}.
$$
Denoting by $H^1_T$ the Sobolev space of $H^1$-functions $x: [0,T] \to \mathbb{R}^2$ satisfying $x(0) = x(T)$, let us define
$$
{\mathcal X}_c=\left\{x\in H^1_T: \ \exists \ t_0\in [0,T] \mbox{ such that } x(t_0)=0\right\},
$$
$$
{\mathcal X}_r=\left\{x\in H^1_T: x(t) \neq 0\; \forall t \in[0,T] \mbox{ and } r_x\neq 0\right\}
$$
and 
$$
{\mathcal X}={\mathcal X}_c\cup {\mathcal X}_r.
$$
It is easy to see that $\mathcal X$ is sequentially weakly closed in $H^1_T$; moreover, in the set $\mathcal X$ a Poincar\'e-type inequality holds true, as proved below (see also \cite{Go-77}).

\begin{proposition} \label{prop-poincare} 
	There exists $K>0$ such that
	\begin{equation} \label{eq-disugpoincare}
	\int_0^T |x(t)|^2\, dt \leq K\, \int_0^T |\dot{x}(t)|^2\,dt,\quad \forall \ x\in {\mathcal X}.
	\end{equation}
\end{proposition}

\begin{proof}
The result is well-known if $x \in {\mathcal X}_c$, since one can write $x(t) = \int_{t_0}^t x(s)\,ds$ (with $x(t_0) = 0$)
and use Cauchy-Scwhartz inequality so as to easily prove \eqref{eq-disugpoincare} with $K = T^2$.

As for ${\mathcal X}_r$, a little more work is needed. For any $x\in {\mathcal X}_r$, we introduce the notation 
\[
x_M=\max_{t\in [0,T]} |x(t)|,\quad x_m=\min_{t\in [0,T]} |x(t)|;
\]
let us observe that $x_m>0$, by definition of ${\mathcal X}_r$. We write $x=\rho e^{i\theta}$, with $\rho = \vert x \vert$, in such a way that
\[
\dot{x}=\dot{\rho} e^{i\theta}+i\rho \dot{\theta}e^{i\theta}
\]
and
$$
|\dot{x}|^2=|\dot{\rho}|^2+\rho^2|\dot{\theta}|^2.
$$
It is immediate so see that
\begin{equation} \label{eq-stimaxquadro}
\int_0^T |x(t)|^2\, dt\leq T\, x_M^2\leq 2T\, (x_M-x_m)^2+2T\, x_m^2;
\end{equation}
moreover, using Cauchy-Schwartz inequality together with elementary estimates, we can obtain
\begin{equation} \label{eq-stimaxpuntoquadro}
\begin{aligned}
\int_0^T |\dot{x}(t)|^2\, dt
&\geq \int_0^T |\dot{\rho}(t)|^2\, dt+x_m^2\, \int_0^T |\dot{\theta}(t)|^2\, dt \\
&\geq \dfrac{1}{T} \left(\int_0^T |\dot{\rho}(t)|\, dt\right)^2+\dfrac{x^2_m}{T} \left(\int_0^T |\dot{\theta}(t)|\, dt\right)^2 \\
&\geq \dfrac{1}{T} \left(x_M-x_m\right)^2+\dfrac{x_m^2}{T} \left(\int_0^T |\dot{\theta}(t)|\, dt\right)^2.
\end{aligned}
\end{equation}
Now, taking into account that $x\in {\mathcal X}_r$, we infer that there exists $k\in \mathbb{Z}$, $k\neq 0$, such that
\[
\theta(T)-\theta(0)=2k\pi,
\]
thus obtaining
\[
\int_0^T |\dot{\theta}(t)|\, dt\geq |\theta(T)-\theta(0)|\geq 2\pi> 1.
\]
From this relation and from \eqref{eq-stimaxpuntoquadro} we deduce that
\begin{equation} \label{eq-stimaxpuntoquadro2}
\begin{array}{l}
\displaystyle \int_0^T |\dot{x}(t)|^2\, dt\geq \dfrac{1}{T}\, \left(x_M-x_m\right)^2+\dfrac{1}{T}\, x_m^2.
\end{array}
\end{equation}
Comparing \eqref{eq-stimaxpuntoquadro2} with \eqref{eq-stimaxquadro}, we plainly conclude that \eqref{eq-disugpoincare} holds true with $K=2T^2$.
\end{proof}

Now, for every $[a,b]\subset [0,T]$ let us define $A_{[a,b]}:{\mathcal X}\to (-\infty, +\infty]$ by
$$
A_{[a,b]}(x)=\int_{a}^{b} \left(\dfrac{1}{2}|\dot{x}(t)|^2+\dfrac{1}{|x(t)|}+ U(t,x(t)) \right)dt,
\quad \forall \ x\in {\mathcal X},
$$
and denote $A_T=A_{[0,T]}$.
From Proposition~\ref{prop-poincare} and assumption \eqref{eq-subquadratic} we deduce that there exist
 $ K' > 0 $ such that:
\[
\begin{aligned}
A_{T}(x) & \ge \int_{0}^{T} \left( \dfrac{|\dot{x}(t)|^{2}}{4}
+ \dfrac{1}{4K} |x(t)|^{2} - C |x(t)|^{\alpha} \right) dt - CT \\
& \ge \int_{0}^{T} \left( \dfrac{|\dot{x}(t)|^{2}}{4}
+ \dfrac{1}{8K} |x(t)|^{2} \right) dt - (K'+C)T,
\end{aligned}
\]
for every $ x \in \mathcal{X} $.
This inequality implies that $ A_{T} $ is coercive on $ \mathcal{X} $ and, therefore, we have the following.
\begin{theorem} \label{teo-minimoglobale}
	There exists $x\in {\mathcal X}$ such that
	\[
	A_T(x)=\min_{y\in {\mathcal X}} A_T(y).
	\]
\end{theorem}

Of course, $x$ is a classical solution of \eqref{eq-kepleromain} if $x\in {\mathcal X}_r$. 

\section{Exploring collisions}\label{sec3}

In this section we assume that the minimum $x$ given by Theorem \ref{teo-minimoglobale} lies in the set ${\mathcal X}_c$ and we to prove that it is a generalized solution of \eqref{eq-kepleromain}, according to Definition \ref{def:bs}.

To this end, we perform a study of the local behavior of $x \in {\mathcal X}_c$ near its collisions. As in condition i) 
of Definition \ref{def:bs}, let
\[
E_x=\{t\in [0,T]: \ x(t)=0\}
\]
the set of collision instants of $x$. From the condition
\[
-\infty<A_T(x)<+\infty
\]
we deduce that $E_x$ has zero measure; taking into account that $E_x$ is closed, by the continuity of $x$, we infer that $[0,T]\setminus E_x$ is the (at most countable) union of open intervals $(a_n,b_n)$, $n\geq 0$, and that $x\in \mathcal{C}^{2}(a_n,b_n)$ satisfies
$$
\ddot{x} = -\dfrac{x}{|x|^3} + \nabla_{x} U(t,x), \quad \forall \ t\in (a_n,b_n),\quad \forall \ n\geq 0.
$$
Defining 
\begin{equation} \label{eq-defenergia}
h_x(t)=\dfrac{1}{2}|\dot{x}(t)|^2-\dfrac{1}{|x(t)|},\quad \forall \ t\in [0,T]\setminus E_x,
\end{equation}
and
$$
I_x(t)=\dfrac{1}{2}|x(t)|^2,\quad \forall \ t\in [0,T],
$$
it is immediate to see that in the open set $[0,T]\setminus E_x$ the so-called virial identity
\begin{equation} \label{eq-LJ}
\ddot{I_x}(t)=\dfrac{1}{|x(t)|}+ U(t,x(t)) +2h_x(t),\quad \forall \ t\in [0,T]\setminus E_x,
\end{equation}
holds true. 

\subsection{The energy function and the number of collisions}

The local study of the energy function $h_x$ defined in \eqref{eq-defenergia} near collisions moves 
from the relation
\[
\int_0^T |h_x(t)|\,dt\leq \int_0^T \left(\dfrac{1}{2}|\dot{x}(t)|^2+\dfrac{1}{|x(t)|}\right)dt =A_T(x)-\int_0^T  U(t,x(t))\, dt,
\]
which implies that $h_x\in L^1(0,T)$. In the next result (following a computation in \cite{Chenciner02} dealing with the autonomous case), we show that the minimality of $x$ implies that $h_x$ can be extended to a continuous function in all $[0,T]$.

\begin{proposition} \label{prop-energia}
	Let $x\in {\mathcal X}_c$ be a minimizer of $A_T$ in $\mathcal X$. Then the energy $h_x$ defined in \eqref{eq-defenergia} belongs to $ W^{1,1}(0,T) $ and, therefore, can be extended to a continuous function in $[0,T]$.
\end{proposition}

\begin{proof}
	We already noted that $h_x\in L^1(0,T)$; hence, we just need to prove that $h_x$ has a distributional derivative which is a $L^1$-function.
	
	To this end, let us fix an arbitrary $\varphi\in \mathcal{C}^\infty_c((0,T))$ and, for $\lambda \in \RR$, define $\psi_\lambda:[0,T]\to \RR$ by
	\begin{equation} \label{eq-defpsilambda}
	\psi_\lambda (t)=t+\lambda \varphi (t),\quad \forall \ t\in [0,T];
	\end{equation}
	since $\varphi$ has compact support in $(0,T)$, we deduce that $\psi_\lambda (0)=0$ and $\psi_\lambda (T)=T$ and the condition
	\[
	\dot{\psi}_\lambda (t)=1+\lambda \dot{\varphi}(t),\quad \forall \ t\in [0,T],
	\]
	implies that there exists $\lambda_\varphi>0$ such that $\psi_\lambda$ is strictly increasing in $[0,T]$, for every $\lambda \in [-\lambda_\varphi, \lambda_\varphi]$. As a consequence, for every $\lambda \in [-\lambda_\varphi,\lambda_\varphi]$ we can define
	$$
	x_\lambda (t)=x(\psi_\lambda(t))=x(t+\lambda \varphi (t)),\quad \forall \ t\in [0,T];
	$$
	from the previous discussion we also get $x_\lambda\in {\mathcal X}_c$, for every $\lambda \in [-\lambda_\varphi,\lambda_\varphi]$.
	Hence, from the minimality of $x$ we deduce that
	$$
		A_T(x)\leq A_T(x_\lambda),\quad \forall \ \lambda \in [-\lambda_\varphi,\lambda_\varphi],
	$$
	thus implying that
	\begin{equation} \label{eq-zeroelle}
	l'(0)=0,
	\end{equation}
	where 
	$$
	l(\lambda)=A_T(x_\lambda)=\int_0^T \left( \dfrac{|\dot{x}_{\lambda}(t)|^2}{2}+\dfrac{1}{|x_\lambda(t)|}+ U(t,x_\lambda(t))\right)dt,\quad \forall \ \lambda \in [-\lambda_\varphi,\lambda_\varphi].
	$$
	
	Our goal now is to compute $l'(0)$. By means of the change of variable $s=\psi_\lambda(t)$, we plainly obtain
		\begin{equation} \label{eq-elle2}
		\begin{aligned}
		&l(\lambda)=\int_0^T \left(\dfrac{\left|\dot{x}(s)\right|^2}{2}\, \dot{\psi_\lambda}(\psi_\lambda^{-1}(s))^2+\dfrac{1}{|x(s)|}+U(\psi_\lambda^{-1}(s),x(s)) \right)\, \dfrac{ds}{\dot{\psi_\lambda}(\psi_\lambda^{-1}(s))}
		\end{aligned}
		\end{equation}
for every $\lambda \in [-\lambda_\varphi,\lambda_\varphi]$. At this point some work is needed to show that it is possible to 
differentiate under the integral sign. Defining
\begin{equation} \label{eq-defglambda}
g_\lambda(s)=\dot{\psi_\lambda}(\psi_\lambda^{-1}(s))=1+\lambda \dot{\varphi}(\psi_\lambda^{-1}(s)),\quad \forall \ s\in [0,T],\quad \lambda \in [-\lambda_\varphi,\lambda_\varphi],
\end{equation}
we obtain
\begin{equation} \label{eq-dim11}
\dfrac{\partial g_\lambda}{\partial \lambda} (s)=\dot{\varphi}(\psi_\lambda^{-1}(s))+\lambda \ddot{\varphi}(\psi_\lambda^{-1}(s))\dfrac{\partial \psi_\lambda^{-1}}{\partial \lambda}(s),\quad \forall \ s\in [0,T],\quad \lambda \in [-\lambda_\varphi,\lambda_\varphi];
\end{equation}
on the other hand, from the relation
\[
\psi_\lambda^{-1}(s)+\lambda \varphi(\psi_\lambda^{-1}(s))=s,\quad \forall \ s\in [0,T],\quad \lambda \in [-\lambda_\varphi,\lambda_\varphi],
\]
we deduce that 
\[
\dfrac{\partial \psi_\lambda^{-1}}{\partial \lambda}(s)+\varphi((\psi_\lambda^{-1}(s))+\lambda \dot{\varphi}((\psi_\lambda^{-1}(s))\dfrac{\partial \psi_\lambda^{-1}}{\partial \lambda}(s)=0,\quad \forall \ s\in [0,T],\quad \lambda \in [-\lambda_\varphi,\lambda_\varphi],
\]
and
\[
\dfrac{\partial \psi_\lambda^{-1}}{\partial \lambda}(s)=-\dfrac{\varphi((\psi_\lambda^{-1}(s))}{1+\lambda \dot{\varphi}((\psi_\lambda^{-1}(s))},\quad \forall \ s\in [0,T],\quad \lambda \in [-\lambda_\varphi,\lambda_\varphi].
\]
Hence, from \eqref{eq-dim11} we obtain
\[
\dfrac{\partial g_\lambda}{\partial \lambda} (s)=\dot{\varphi}(\psi_\lambda^{-1}(s))-\lambda \dfrac{\ddot{\varphi}(\psi_\lambda^{-1}(s))\, \varphi((\psi_\lambda^{-1}(s))}{1+\lambda \dot{\varphi}((\psi_\lambda^{-1}(s))},\quad \forall \ s\in [0,T],\quad \lambda \in [-\lambda_\varphi,\lambda_\varphi];
\]
therefore, taking again into account \eqref{eq-defglambda}, we deduce that there exist $\lambda'_\varphi\leq \lambda_\varphi$, $M\in (0,1)$ and $M'>0$ such that
\begin{equation} \label{eq-dim12}
|g_\lambda (s)|\geq M,\quad \left|\dfrac{\partial g_\lambda}{\partial \lambda} (s)\right|\leq M',\quad \forall \ s\in [0,T],\quad \lambda \in [-\lambda'_\varphi,\lambda'_\varphi].
\end{equation}
Now, recalling \eqref{eq-defpsilambda}, a simple computation shows that
$$
\begin{aligned}
\dfrac{\partial}{\partial \lambda} & \left(\dfrac{\left|\dot{x}(s)\right|^2}{2}\, \dot{\psi_\lambda}(\psi_\lambda^{-1}(s))+\dfrac{1}{|x(s)|\dot{\psi_\lambda}(\psi_\lambda^{-1}(s))}+ \dfrac{U(\psi_\lambda^{-1}(s),x(s)) }{\dot{\psi_\lambda}(\psi_\lambda^{-1}(s))}\right) \\
&=\dfrac{\left|\dot{x}(s)\right|^2}{2}\, \dfrac{\partial g_\lambda}{\partial \lambda} (s)-\dfrac{1}{|x(s)|}\, \dfrac{\dfrac{\partial g_\lambda}{\partial \lambda}(s)}{g_\lambda(s)^2}
+ \partial_{t} U(\psi_\lambda^{-1}(s),x(s)) \, \dfrac{\partial \psi_\lambda^{-1}}{\partial \lambda}(s)\, \dfrac{1}{g_\lambda(s)} \\
& \hphantom{=}-U(\psi_\lambda^{-1}(s),x(s)) \, \dfrac{\dfrac{\partial g_\lambda}{\partial \lambda}(s)}{g_\lambda(s)^2},
\end{aligned}
$$
for every $s\in [0,T]$ and $\lambda \in [-\lambda'_\varphi,\lambda'_\varphi]$; from \eqref{eq-dim12}, setting
\[
M_1=\max_{t\in [0,T]} |\varphi(t)|,\quad M_2=\max_{t\in [0,T]} |U(t,x(t))|,\quad
M_3=\max_{t\in [0,T]} |\partial_{t}U(t,x(t))|
\] 
we deduce that
\begin{equation} \label{eq-stimaderivata}
\begin{aligned}
\Bigg|\dfrac{\partial}{\partial \lambda} &\left(\dfrac{\left|\dot{x}(s)\right|^2}{2}\, \dot{\psi_\lambda}(\psi_\lambda^{-1}(s))+
\dfrac{1}{|x(s)|\dot{\psi_\lambda}(\psi_\lambda^{-1}(s))}+ \dfrac{U(\psi_\lambda^{-1}(s),x(s))}{\dot{\psi_\lambda}(\psi_\lambda^{-1}(s))}\right)\Bigg| \\
&\leq \dfrac{M'}{M^2} \, \left(\dfrac{\left|\dot{x}(s)\right|^2}{2}+\dfrac{1}{|x(s)|}\right)+\left(\dfrac{M_1 M_3}{M^2}+\dfrac{M' M_2}{M^2}\right),
\end{aligned}
\end{equation}
for every $s\in [0,T]$ and $\lambda \in [-\lambda'_\varphi,\lambda'_\varphi]$.
Observing that the right-hand side in \eqref{eq-stimaderivata} is an integrable function in $[0,T]$, from \eqref{eq-elle2} and \eqref{eq-stimaderivata} we infer that
\[
\begin{aligned}
\displaystyle l'(\lambda) = \int_0^T \Bigg(&\dfrac{\left|\dot{x}(s)\right|^2}{2}\, \dfrac{\partial g_\lambda}{\partial \lambda} (s)-\dfrac{1}{|x(s)|}\, \dfrac{\dfrac{\partial g_\lambda}{\partial \lambda}(s)}{g_\lambda(s)^2}
\\
& + \partial_{t}U(\psi_\lambda^{-1}(s),x(s)) \, \dfrac{\partial \psi_\lambda^{-1}}{\partial \lambda}(s)\, \dfrac{1}{g_\lambda(s)}-U(\psi_\lambda^{-1}(s),x(s)) \, \dfrac{\dfrac{\partial g_\lambda}{\partial \lambda}(s)}{g_\lambda(s)^2} \Bigg)ds,
\end{aligned}
\]
for every $\lambda \in [-\lambda'_\varphi,\lambda'_\varphi]$, and in particular, also integrating by parts, 
\[
\begin{aligned}
l'(0) & = \int_{0}^{T} \left[ \left(\dfrac{\left|\dot{x}(s)\right|^2}{2} -\dfrac{1}{|x(s)|}-U(s,x(s))\right)\dot{\varphi}(s)
- \partial_{t}U(s,x(s))\varphi(s) \right] ds\\
&=\int_0^T \left[ h_x(s)\dot{\varphi}(s)+ \langle \nabla_{x}U(s,x(s)),\dot{x}(s) \rangle \varphi(s) \right] ds.
\end{aligned}
\]
Recalling \eqref{eq-zeroelle}, we conclude that
\[
\int_0^T \left[ h_x(s)\dot{\varphi}(s)+ \langle \nabla_{x}U(s,x(s)),\dot{x}(s) \rangle \varphi(s) \right] ds =0,
\]
for every $\varphi\in \mathcal{C}^\infty_c((0,T))$; this shows 
\[
\dot{h}_{x}(t)=-\langle \nabla_{x}U(t,x(t)),\dot{x}(t) \rangle,
\]
in the distributional sense.
To conclude the proof, it is sufficient to observe that $ U \in \mathcal{C}^{1}$
and $\dot{x}\in L^2(0,T)$.
\end{proof}

From Proposition \ref{prop-energia} we deduce that $h_x$ is bounded in $[0,T]$; using this fact, arguing exactly as in \cite[Lemma 3]{SoaTer13}, from \eqref{eq-LJ} we conclude that collisions are isolated, implying that $E_x$ is a finite set.
Moreover, the continuity of $ h_{x} $ also implies that the second limit in condition (iii) of Definition~\ref{def:bs}
exists and is finite.

\subsection{Asympotic directions near isolated collisions}

In this part, via a blow-up analysis, we study the local behaviour of the ratio $x/|x|$ near a collision of $x$; to this end, we use the classical asymptotic estimates near collisions due to Sperling (see \cite{Spe69}), together with the comparison between action levels of solutions of the unperturbed Kepler problem.

Let $t_0\in (0,T)$ be a collision of $x$; if $t_0=0$ or $t_0=T$, the argument is the same, replacing $[0,T]$ by $[-T/2,T/2]$, by periodicity.

From the previous discussion we know that $t_0$ is isolated; as a consequence, there exists $\bar t>0$ such that
\[
x(t)\neq 0,\quad \forall \ t\in \bar I:=[t_0-\bar t,t_0+\bar t],\ t\neq t_0.
\]
From the classical paper by Sperling \cite{Spe69} it is known that there exist $x_0^\pm\in \RR^2$, with $|x_0^\pm|=1$, such that
\begin{equation} \label{eq-stimesperling11}
\begin{array}{l}
x(t)=\sqrt[3]{\dfrac{9}{2}}\, (t-t_0)^{2/3}\, x_0^+ +R^+(t),\quad \forall \ t\in [t_0,t_0+\bar t],\\
\\
x(t)=\sqrt[3]{\dfrac{9}{2}}\, (t-t_0)^{2/3}\, x_0^-+R^-(t),\quad \forall \ t\in [t_0-\bar t,t_0],
\end{array}
\end{equation}
and
\begin{equation} \label{eq-stimesperling12}
\begin{array}{l}
\displaystyle \dot{x}(t)=\dfrac{2}{3}\, \sqrt[3]{\dfrac{9}{2}}\, (t-t_0)^{-1/3}\, x^+_0+\dot{R}^+(t),\quad \forall \ t\in (t_0,t_0+\bar t],\\
\\
\displaystyle \dot{x}(t)=\dfrac{2}{3}\, \sqrt[3]{\dfrac{9}{2}}\, (t-t_0)^{-1/3}\, x^-_0+\dot{R}^-(t),\quad \forall \ t\in [t_0-\bar t,t_0),
\end{array}
\end{equation}
for some $R^\pm \in \mathcal{C}(\bar I;\RR^2)\cap \mathcal{C}^1([t_0-\bar t,t_0)\cup (t_0,t_0+\bar t];\RR^2)$ such that 
\begin{equation} \label{eq-stimesperling13}
\lim_{t\to t_0^\pm} \dfrac{R^\pm(t)}{(t-t_0)^{2/3}}=0 \qquad \text{and} \qquad
\lim_{t\to t_0^\pm} \dfrac{\dot{R}^\pm(t)}{(t-t_0)^{-1/3}}=0.
\end{equation}
As a consequence, there exists $C_x>0$ such that
\begin{equation} \label{eq-stimesperling2}
\begin{array}{l}
|x(t)|\leq C_x (t-t_0)^{2/3},\quad \forall \ t\in \bar I,\\
\\
\displaystyle |\dot{x}(t)|\leq \dfrac{C_x}{|t-t_0|^{1/3}},\quad \forall \ t\in \bar I,\ t\neq t_0.
\end{array}
\end{equation}
Now, since by Proposition \ref{prop-energia} $h_x$ can be extended to a continuous function in $t_0$, from \eqref{eq-LJ} we deduce that $I$ is strictly convex in a neighborhood of $t_0$; hence, for every sufficiently small $\delta>0$ there exist unique $ 0<t_\delta^\pm<\bar t$ such that
\begin{equation} \label{eq-convessita}
\begin{array}{l}
|x(t_0-t_\delta^-)|=|x(t_0+t_\delta^+)|=\delta,\\
\\
\displaystyle |x(t)|<\delta,\quad \forall \ t\in I_\delta:=(t_0-t_\delta^-,t_0+t_\delta^+).
\end{array}
\end{equation}

We observe that estimates \eqref{eq-stimesperling11} already imply that
\begin{equation}\label{eq-direzionicollisione}
\lim_{ t \to t_{0}^{\pm} } \frac{ x(t) }{ |x(t)| } = x_{0}^{\pm}.
\end{equation}
Therefore, in order to conclude our proof and show that the
minimum $ x $ of the action functional on $ \mathcal{X} $ is a generalized solution in the sense of
Definition~\ref{def:bs}, we just need to show that $ x_{0}^{+} = x_{0}^{-} $ 
This last fact will be obtained by showing that a collision solution with $ x_{0}^{+} \ne x_{0}^{-} $ cannot be a minimizer
of the action in $ \mathcal{X} $: in fact, we will prove that, if $ x_{0}^{+} \ne x_{0}^{-} $, then it is possible to 
modify $ x $ in a neighborhood of the collision time $ t_{0} $ and to obtain a non-collision path with a smaller
action that still belongs to $ \mathcal{X} $.

The first step in this argument is an estimate from below of $ A_{[t_0-t_\delta^-,t_0+t_\delta^+]} (x) $.
The comparison term involves the action relative to colliding parabolic Keplerian orbits.
More precisely, given $x_{0}^\pm \in \RR^2$, with $|x_{0}^\pm|=1$, let us define
\begin{equation} \label{eq-defomotetica}
\zeta_0(t;x_{0}^-,x_{0}^+)=\left\{\begin{array}{ll}
\displaystyle \sqrt[3]{\dfrac{9}{2}}\,t^{2/3} x_{0}^+&\mbox{ if $t\geq 0$,}\\
&\\
\displaystyle \sqrt[3]{\dfrac{9}{2}}\, t^{2/3} x_{0}^-&\mbox{ if $t< 0$.}
\end{array}
\right.
\end{equation}
It is easy to check that $\zeta_0(t;x_{0}^-,x_{0}^+)$ is a parabolic solution of the unperturbed Kepler equation in the
intervals $(-\infty,0)$ and $(0,+\infty)$, having a collision at $t=0$; moreover, taking $s_0 > 0$ such that
\begin{equation}\label{eq-defessezero}
\sqrt[3]{\dfrac{9}{2}\, s_0^{2}} = 1 ,
\end{equation}
it holds that
\[
\left| \zeta_0(\pm s_0;x_{0}^-,x_{0}^+) \right| = 1.
\]
We then define
\begin{equation}\label{eq-azionecollisione}
\varphi_0=\int_{-s_0}^{s_0}\left( \frac{1}{2} \left|\dot{\zeta}_0(t;x_{0}^-,x_{0}^+)\right|^2
+\dfrac{1}{\left|\zeta_0(t;x_{0}^-,x_{0}^+)\right|} \right)\,dt = 4\sqrt[6]{8},
\end{equation}
which actually does not depend on $x_{0}^\pm$ and is the action of $ \zeta_0(\cdot;x_{0}^-,x_{0}^+) $ in $ [-s_{0},s_{0}] $
relatively to Kepler problem without forcing term.
Then, we are able to prove the following estimate.
\begin{lemma} \label{prop-livellocollisione}
Let $t_{0}$ be a collision time for a minimizer $x$ of $A_{T}$ in $\mathcal{X}$ and $ t^{\pm}_{\delta} $,
$ x^{\pm}_{0} $, $ \zeta_{0} $, $ s_{0} $ and $ \varphi_{0} $ be as in \eqref{eq-convessita}--\eqref{eq-azionecollisione}.
If we set $\sigma^\pm_\delta=t_\delta^\pm / \delta^{3/2}$, then we have that
	\begin{align*}
    &\lim_{\delta \to 0^{+} } \sigma_{\delta}^{\pm} = s_{0},\\
	&\liminf_{\delta \to 0^+} \dfrac{A_{[t_0-t_\delta^-,t_0+t_\delta^+]} (x)}{\delta^{1/2}} \geq \varphi_0.
	\end{align*}
\end{lemma}
\begin{proof}
We employ a blow-up argument: for every $\delta>0$ we define
	\begin{equation} \label{eq-defzetadelta}
	z_\delta(t)=\dfrac{1}{\delta} x\left(\delta^{3/2}t+t_0\right),
	\end{equation}
	for every $t\in J_\delta :=[-\sigma^-_\delta,\sigma_\delta^+]$.
From the relations in \eqref{eq-convessita}, recalling also that $x$ has a collision in $t_0$, we deduce that the function $z_\delta$ satisfies
\begin{equation} \label{eq-proprzetadelta}
\begin{array}{l}
|z_\delta (-\sigma^-_\delta)|=|z_\delta (\sigma^+_\delta)|=1,\quad z_\delta(0)=0,\\
\\
\displaystyle |z_\delta(t)|\leq 1,\quad \forall \ t\in J_\delta;
\end{array}
\end{equation}
moreover, conditions \eqref{eq-stimesperling2} imply that
$$
\begin{array}{l}
|z_\delta(t)|\leq C_x t^{2/3},\quad \forall \ t\in J_\delta,\\
\\
\displaystyle |\dot{z}_\delta(t)|\leq \dfrac{C_x}{|t|^{1/3}},\quad \forall \ t\in J_\delta,\ t\neq 0,
\end{array}
$$
being
\begin{equation} \label{eq-derzetadelta}
\begin{array}{l}
\dot{z}_\delta (t)=\delta^{1/2} \, \dot{x}(\delta^{3/2}t+t_0),\quad \forall \ t\in J_\delta,\ t\neq 0.
\end{array}
\end{equation}
Let us also observe that from \eqref{eq-defzetadelta} and the first relation in \eqref{eq-derzetadelta} we deduce 
\begin{equation} \label{eq-azionezetadelta1}
\begin{aligned}
&\dfrac{A_{[t_0-t_\delta^-,t_0+t_\delta^+]} (x)}{\delta^{1/2}} =
\dfrac{1}{\delta^{1/2}} \int_{t_0-t^-_\delta}^{t_0+t^+_\delta} \left( \dfrac{|\dot{x}(t)|^2}{2} + \dfrac{1}{|x(t)|} +
U(t,x(t)) \right) dt \\
\\
&\hspace{40pt}=
\int_{-\sigma^-_\delta}^{\sigma^+_\delta} \left( \dfrac{|\dot{z}_\delta(t)|^2}{2} +\dfrac{1}{|z_\delta(t)|} \right) dt
+\delta^{2} \int_{-\sigma^-_\delta}^{\sigma^+_\delta} U\left( \delta^{3/2}t+t_0, z_\delta(t) \right) dt.
\end{aligned}
\end{equation}
Now, let us study the convergence of the sequence $z_\delta$ when $\delta \to 0^+$.
Since
\[
z_\delta(\sigma^\pm)=\int_0^{\sigma^\pm_\delta} \dot{z}_\delta(t)\, dt,
\]
from  \eqref{eq-stimesperling2} and \eqref{eq-proprzetadelta} we deduce that
\[
1\leq \int_{0}^{\sigma^\pm_\delta} \dfrac{C_x}{t^{1/3}}\,dt=\dfrac{3}{2}\, C_x \, (\sigma^\pm_\delta)^{2/3},
\]
which shows that $\sigma^\pm_\delta$ are bounded away from zero, thus implying
\begin{equation}\label{defliminf}
\liminf_{\delta \to 0^+} \sigma^\pm_\delta =: \sigma^{\pm} > 0.
\end{equation}
Now, from \eqref{eq-stimesperling11} and \eqref{eq-stimesperling12} we infer that 
$$
\begin{array}{l}
z_\delta(t)=\sqrt[3]{\dfrac{9}{2}}\, t^{2/3}\, x_0^+ +\dfrac{1}{\delta}\, R^+(\delta^{3/2}t+t_0),\quad \forall \ t\in [0,\sigma^+_\delta],\\
\\
z_\delta(t)=\sqrt[3]{\dfrac{9}{2}}\, t^{2/3}\, x_0^-+\dfrac{1}{\delta}\, R^-(\delta^{3/2}t+t_0),\quad \forall \ t\in [\sigma^-_\delta,0],
\end{array}
$$
and
$$
\begin{array}{l}
\displaystyle \dot{z}_\delta(t)=\dfrac{2}{3}\, \sqrt[3]{\dfrac{9}{2}}\, t^{-1/3}\, x^+_0+\delta^{1/2} \dot{R}^+(\delta^{3/2}t+t_0),\quad \forall \ t\in (0,\sigma^+_\delta],\\
\\
\displaystyle \dot{z}_\delta(t)=\dfrac{2}{3}\, \sqrt[3]{\dfrac{9}{2}}\, t^{-1/3}\, x^-_0+\delta^{1/2} \dot{R}^-(\delta^{3/2}t+t_0),\quad \forall \ t\in [\sigma^-_\delta,0),
\end{array}
$$
where, taking into account \eqref{eq-stimesperling13}, 
$$
\begin{array}{l}
\displaystyle \lim_{\delta\to 0^+} \dfrac{1}{\delta}\, R^\pm (\delta^{3/2}t+t_0)=\lim_{s\to t_0} t^{2/3} \, \dfrac{R^\pm(s)}{(s-t_0)^{2/3}}=0,\\
\\
\displaystyle \lim_{\delta\to 0^+} \delta^{1/2} \dot{R}^\pm (\delta^{3/2}t+t_0)=\lim_{s\to t_0} \dfrac{1}{t^{1/3}} \, \dfrac{R^\pm(s)}{(s-t_0)^{-1/3}}=0,
\end{array}
$$
for every $t\in (-\sigma^-,\sigma^+)$, $t\neq 0$.
Therefore, we obtain 
\[
\begin{array}{l}
\displaystyle \lim_{\delta \to 0^+} z_\delta(t)=\sqrt[3]{\dfrac{9}{2}}\, t^{2/3}\, x_0^+,\quad \forall \ t\in [0,\sigma^+),\\
\\
\displaystyle \lim_{\delta \to 0^+} z_\delta(t)=\sqrt[3]{\dfrac{9}{2}}\, t^{2/3}\, x_0^-,\quad \forall \ t\in (-\sigma^-, 0],
\end{array}
\]
and
\[
\begin{array}{l}
\displaystyle \lim_{\delta \to 0^+} \dot{z}_\delta(t)=\dfrac{2}{3}\, \sqrt[3]{\dfrac{9}{2}}\, t^{-1/3}\, x^+_0,\quad \forall \ t\in (0,\sigma^+),\\
\\
\displaystyle \lim_{\delta \to 0^+} \dot{z}_\delta(t)=\dfrac{2}{3}\, \sqrt[3]{\dfrac{9}{2}}\, t^{-1/3}\, x^-_0,\quad \forall \ t\in (-\sigma^-, 0).
\end{array}
\]
In particular, $z_\delta(t)$ converges pointwise to $ \zeta_{0}(t;x_{0}^{-},x_{0}^{+}) $ for all $ t\in (-\sigma^-, \sigma^+) $;
recalling \eqref{eq-defessezero} and \eqref{eq-proprzetadelta}, we then deduce that the inferior limit in \eqref{defliminf} is actually a limit and
\[
\sigma^+=\sigma^-=s_0.
\]
Using again the pointwise convergence of $z_\delta(t)$ and $\dot z_{\delta}(t)$, we can use Fatou's lemma in \eqref{eq-azionezetadelta1}, thus obtaining
\[
\liminf_{\delta \to 0^+} \dfrac{A_{[t_0-t_\delta^-,t_0+t_\delta^+]} (x)}{\delta^{1/2}} \geq \varphi_0.
\qedhere
\]
\end{proof}

The previous blow-up argument shows that a suitable rescaling of a minimizer around a collision time converges to the
parabolic collision solution of Kepler's problem $\zeta_{0}(\cdot;x_{0}^{-},x_{0}^{+})$, given by \eqref{eq-defomotetica}, that joins two points $x_{0}^{-}$ in a prescribed time interval of length $2 s_0$.
However, if $ x_{0}^{+} \ne x_{0}^{-} $, there exist two collision-free Keplerian orbits,
$ \xi_{0}^{1} $ and $ \xi_{0}^{2} $, that joins the same two
points in the same time interval but with a smaller action than $ \varphi_{0} $.
In the next result, we collect this fact in the next result together with other useful and known properties.
\begin{lemma} \label{prop-archi} \cite[Prop. 5.7]{FuGrNe}
For any $x_{0}^{\pm} \in \RR^2$, with $|x_{0}^{\pm}|=1$ and $x_{0}^{+}\neq x_{0}^{-}$, there are
two solutions $\xi^i=\xi^i(\cdot; x_{0}^{-},x_{0}^{+}):[-s_0,s_0] \to \mathbb{R}^2 \setminus \{(0,0)\}$, $i=1,2$,
of the problem
\begin{equation} \label{eq-keplero}
\begin{cases}
\ddot{\xi} = -\dfrac{\xi}{|\xi|^3} & \text{in } [-s_{0},s_{0}]\\
\xi(\pm s_{0})=x_{0}^{\pm}
\end{cases}
\end{equation}
such that
\begin{itemize}
\item[i)]
they parametrize two simple curves which are not homotopic to each other in $ \mathbb{R}^{2} \setminus \{(0,0)\} $
with fixed endpoints;
\item[ii)]
their actions satisfy
\[
\varphi_0^i (x_{0}^{-},x_{0}^{+}) :=
\int_{-s_0}^{s_0} \left( \frac{|\dot{\xi}^i(t)|^{2}}{2}+\dfrac{1}{|\xi^i(t)|} \right)\,dt < \varphi_{0},\qquad i=1, 2;
\]
\item[iii)]
up to a suitable choice of the label $i\in\{1,2\}$, they depend smoothly on $ x_{0}^{\pm} $;
namely, if $ x_{n}^{\pm} \to x_{0}^{\pm} $
as $ n \to +\infty $, with $|x_{0}|^{\pm} = 1 $ and $ x_{0}^{+} \ne x_{0}^{-} $, then
$ \xi^{i}(\cdot;x_{n}^{-},x_{n}^{+}) \to \xi^i(\cdot; x_{0}^{-},x_{0}^{+}) $ in $ \mathcal{C}^{2}([-s_{0},s_{0}];\mathbb{R}^{2})$.
\end{itemize}
\end{lemma}
The above solutions $ \xi^{1} $ and $ \xi^{2} $  are usually called the \emph{direct} and \emph{indirect} Keplerian arcs
and the proof of their existence is typically attributed to Marchal (see \cite[Section 5.2]{Chenciner05}).
Nowadays, various proofs are available, at different level of generality (see \cite{Che,FuGrNe,SoaTer13,TeVe,Yu16}).
In particular, their existence with the first statement is proved in \cite{Alb2002}, while the estimate
in the second statement is considered in \cite[Proposition~5.7]{FuGrNe}.
The third statement follows from the theorem of continuous dependence on initial data
as soon as one realizes that the initial speed $\dot{\xi}(-s_{0})$ of the solutions of \eqref{eq-keplero} depend smoothly
on $ x_{0}^{\pm} $.
It is possible to use Lambert's Theorem (see \cite[Lecture 5]{Alb2002})
to find an explicit formula that links the energy $ H_{i} $ of $ \xi^{i} $ and $ | x_{0}^{+} - x_{0}^{-} | $.
Indeed, Lambert's Theorem states that the quantities $ \ell := |\xi(-s_{0})| + |\xi(s_{0})| = 2 $,
$ \Delta t := 2s_{0} = 2 \sqrt{2} / 3 $, $ c:= | \xi(s_{0}) - \xi(-s_{0}) | $ and $ H := |\dot{\xi}|^{2} - 1/|\xi| $
are functionally dependent for the solutions of \eqref{eq-keplero} and their functional relation is the same
for all configurations of $ \xi(\pm s_{0}) $ as long as $ \ell $ and $ c $ are kept constant.
Therefore one obtains easily that the energy $H$ and, hence, the modulus of the initial speed $ |\dot{\xi}(-s_{0})| $
depend continuously on $ x_{0}^{\pm} $.
As for the continuity of the initial speed versor $ \dot{\xi}(-s_{0})/|\dot{\xi}(-s_{0})| $ one can use the arguments
and the parametrization of the orbit of the solutions of \eqref{eq-keplero} given in \cite[Appendix 2]{FuGrNe}.

We can now conclude our argument by showing that the asymptotic directions at a collision for a minimum of
the action cannot be different.
\begin{proposition}\label{pro-rimbalzo}
Let $t_{0}$ be a collision time for a minimizer $x$ of $A_{T}$ in $\mathcal{X}$. Then $ x_{0}^{+} = x_{0}^{-} $.
\end{proposition}
\begin{proof}
Let $ q_{\delta}^{\pm} = x( t_{0} \pm t_{\delta}) / |x( t_{0} \pm t_{\delta})| $.
If, by contradiction, we suppose that $ x_{0}^{+} \ne x_{0}^{-} $, then we have
$ q_{\delta}^{+} \ne q_{\delta}^{-} $ for all $ \delta > 0 $ sufficiently small.
Thus, we can apply Lemma~\ref{prop-archi} and use the Keplerian arcs $ \xi^{i}(\cdot,q_{\delta}^{-},q_{\delta}^{+}) $
to modify $ x $ in a neighborhood of $ t_{0} $  and to obtain two different paths in the following way:
\[
y_{\delta}^{i}(t) =
\begin{cases}
x(t) & \text{if } t \in \left[0,t_{0}-t_{\delta}^{-}\right)\cup\left(t_{0}+t_{\delta}^{+},T\right], \\
\delta \xi^{i}\left( \dfrac{t-t_{0}}{\delta^{3/2}}; q_{\delta}^{-},q_{\delta}^{+} \right)
& \text{if } t \in [ t_{0}-t_{\delta}^{-}, t_{0}+t_{\delta}^{+} ],
\end{cases}
\text{ for } i=1,2.
\]
Thanks to Statement i) in Lemma~\ref{prop-archi}, for each $ \delta >0$ at least one between $ y^{1}_{\delta} $ and
$ y^{2}_{\delta} $ belongs to $ \mathcal{X} $.
Straightforward computations show that:
\[
\begin{aligned}
\dfrac{ A_{[t_{0}-t_{\delta}^{-},t_{0}+t_{\delta}^{+}]}(y_{\delta}^{i}) }{\delta^{1/2}} = 
& \int_{-\sigma^-_\delta}^{\sigma^+_\delta} \left( \dfrac{|\dot{\xi}^{i}(s;q_{\delta}^{-},q_{\delta}^{+})|^2}{2} +\dfrac{1}{|\xi^{i}(s;q_{\delta}^{-},q_{\delta}^{+})|} \right) ds \\
& +\delta^{2} \int_{-\sigma^-_\delta}^{\sigma^+_\delta}
U\left(\delta^{3/2}s+t_0, \xi^{i}(s;q_{\delta}^{-},q_{\delta}^{+}) \right) ds.
\end{aligned}
\]
By Lemma ~\ref{prop-livellocollisione} and Statements ii) and iii) of Lemma~\ref{prop-archi}, we deduce that
\[
\lim_{\delta \to 0^{+}} \dfrac{ A_{[t_{0}-t_{\delta}^{-},t_{0}+t_{\delta}^{+}]}(y_{\delta}^{i}) }{\delta^{1/2}} =
\varphi_{0}^{i}(x_{0}^{-},x_{0}^{+}) < \varphi_{0}
\le \liminf_{\delta\to 0^{+}} \dfrac{A_{[t_0-t_\delta^-,t_0+t_\delta^+]} (x)}{\delta^{1/2}}, \quad i = 1,2.
\] 
Therefore, we have that $ A_{T}(y_{\delta}^{i}) < A_{T}(x) $ for $ i=1,2 $ and every $ \delta $ sufficiently small,
which contradicts the minimality of $ x $.
\end{proof}

\end{document}